\documentclass[11pt,a4paper,reqno]{amsart}
\usepackage{amsmath,amssymb,mathtools,mathrsfs}
\usepackage{booktabs,tabularx,array}
\usepackage{placeins}
\usepackage[dvipsnames]{xcolor}
\usepackage{microtype,cite,enumitem}
\usepackage[pagebackref=true,colorlinks=false]{hyperref}
\hypersetup{pdfborder={0 0 0.65},linkbordercolor={1 0 0},citebordercolor={0 0.65 0},urlbordercolor={0 0 1},pdftitle={Lacunary maximal operators on convex boundaries: motivated by Wright's Problem 16},pdfauthor={Xi Cen, Guixiang Hong, Xinfeng Wu}}
\allowdisplaybreaks[2]
\newcolumntype{Y}{>{\raggedright\arraybackslash}X}
\newcommand{\fullref}[2]{\mbox{\hyperref[#1]{#2~\ref*{#1}}}}
\newcommand{\Thmitemref}[2]{\mbox{\hyperref[#1]{Theorem~\ref*{#1}#2}}}
\newcommand{\Propitemref}[2]{\mbox{\hyperref[#1]{Proposition~\ref*{#1}#2}}}
\newcommand{\Eqref}[1]{\hyperref[#1]{\textup{(\ref*{#1})}}}
\newcommand{\Thmref}[1]{\fullref{#1}{Theorem}}
\newcommand{\Lemref}[1]{\fullref{#1}{Lemma}}

\newcommand{\Probref}[1]{\fullref{#1}{Problem}}
\newcommand{\Secref}[1]{\fullref{#1}{Section}}

\newcommand{\doilink}[1]{\mbox{\href{https://doi.org/#1}{doi:\nolinkurl{#1}}}}
\newcommand{\arxiv}[1]{\mbox{\href{https://arxiv.org/abs/#1}{arXiv\,\nolinkurl{#1}}}}
\renewcommand*{\backref}[1]{}

\renewcommand*{\backrefalt}[4]{\ifcase #1\relax\or~{\footnotesize #2}\else~{\footnotesize #2}\fi}

\numberwithin{equation}{section}
\newtheoremstyle{boldplain}{6pt}{6pt}{\itshape}{}{\bfseries}{.}{.5em}{}
\newtheoremstyle{boldremark}{6pt}{6pt}{\normalfont}{}{\bfseries}{.}{.5em}{}
\theoremstyle{boldplain}
\newtheorem{theorem}{Theorem}[section]
\newtheorem{lemma}{Lemma}[section]

\newtheorem*{theoremA}{Theorem A}
\theoremstyle{boldremark}
\newtheorem{problem}{Problem}[section]
\newtheorem{remark}[theorem]{Remark}

\newcommand{\R}{\mathbb R}
\newcommand{\Z}{\mathbb Z}

\newcommand{\dd}{\,\mathrm d}

\newcommand{\one}{\mathbf1}

\DeclareMathOperator{\dist}{dist}
\DeclareMathOperator{\Var}{Var}

\begin{document}
\title[Convex lacunary maximal operators]{\bf Lacunary maximal operators on convex boundaries}
\author[X. Cen]{Xi Cen}
\address{Xi Cen, School of Science, China University of Mining and Technology-Beijing, Beijing 100083, People's Republic of China}
\email{xicenmath@gmail.com}
\author[X. Wu]{Xinfeng Wu}
\address{Xinfeng Wu, School of Science, China University of Mining and Technology-Beijing, Beijing 100083, People's Republic of China}
\email{wuxf@cumtb.edu.cn}
\date{}

\begin{abstract}
We disprove the cap-sufficiency proposal in Wright's Problem 16 for
$1<p<(\sqrt{13}-1)/2$ and for $p>3$. For each fixed exponent in these
ranges, one smooth strictly convex
domain has finite original directional cap integral, while the lacunary
maximal operator of complete arclength fails restricted weak type.
We also give a geometric sufficient condition for strong $L^p$
boundedness, for every finite $1<p<\infty$, on an arbitrary given
bounded convex domain. It requires one logarithmic integral of a curvature lower envelope,
without a differentiability assumption on the boundary or a restriction
on the number of curvature zeros.
Both results concern unnormalised true arclength, one fixed interior
origin and all integer dyadic dilations.

\end{abstract}
\maketitle
\begin{center}\begin{minipage}{12.3cm}\small
\end{minipage}\end{center}

\section{Introduction}\label{sec:introduction}
\subsection{Background and motivations}\label{subsec:background}
Let $\Omega\subseteq\R^2$ be a bounded convex domain with
$0\in\operatorname{int}\Omega$, and let
$\sigma=\mathcal H^1|_{\partial\Omega}$ be its complete unnormalised
arclength. We study
\begin{equation}\label{eq:original-maximal}
 M_\sigma f(x)=\sup_{k\in\Z}
 \left|\int_{\partial\Omega}f(x-2^ky)\dd\sigma(y)\right|.
\end{equation}
The domain and the origin are fixed. The selected scale may depend
on $x$, so a single-average estimate does not determine the maximal
bound. A classical starting point for this dyadic-dilation problem is
lacunary spherical averaging \cite{Calderon}.

Fourier-decay criteria for maximal convolutions were developed by
Duoandikoetxea and Rubio de Francia \cite{DR}. A related
measure-theoretic approach on homogeneous groups is presented in
\cite{GHW}. Seeger and Wright
\cite{SW} formulated a geometric condition for convex boundaries
in terms of supporting caps. Let $\ell^\pm(e)$ be the supporting
lines perpendicular to $e\in S^1$. For a fixed small $\delta_*>0$, set
\begin{equation}\label{eq:original-caps}
 \begin{split}
 \Lambda_\Omega(e,\delta)
 &=\max_\pm\sigma\{y\in\partial\Omega:
                         \dist(y,\ell^\pm(e))\le\delta\},\\
 \mathcal C_p(\Omega)
 &=\sup_{e\in S^1}\int_0^{\delta_*}
             \Lambda_\Omega(e,\delta)^p\frac{\dd\delta}{\delta}.
 \end{split}
\end{equation}
The supporting direction is fixed throughout the depth integral.
Convex cap estimates have a classical finite-type counterpart in
\cite{BNW}. For the general planar setting, the complete-boundary estimate \cite[(4), p.~238]{SW},
\begin{equation}\label{eq:published-cap-fourier}
 |\widehat\sigma(Re)|\lesssim
       \Lambda_\Omega(e,R^{-1})\qquad(R\ge R_\Omega),
\end{equation}
connects shallow-cap geometry with the Fourier analysis of the
original dilates. The following theorem identifies the exact Hilbert-space
criterion and supplies the necessary condition used throughout this paper.

\begin{theoremA}[Seeger--Wright
{\cite[Theorem~1.2, (5)--(6), p.~238]{SW}}]
\phantomsection\label{thm:published-hilbert}
For every bounded convex planar domain containing the origin,
\begin{equation}\label{eq:published-hilbert}
 \|M_\sigma\|_{L^2\to L^2}<\infty
       \quad\Longleftrightarrow\quad\mathcal C_2(\Omega)<\infty.
\end{equation}
For every $1<p<\infty$, Seeger and Wright also proved the necessary
estimate
\begin{equation}\label{eq:published-necessary}
       \mathcal C_p(\Omega)^{1/p}
                   \lesssim\|M_\sigma\|_{L^p\to L^p}.
\end{equation}
\end{theoremA}

Wright's contribution to \cite[Section~10, pp.~15--16]{Problems}
asks for a geometric description of strong boundedness.
\begin{problem}[Wright
{\cite[Problem~16, pp.~15--16]{Problems}}]\label{prob:wright}
For each fixed $1<p<\infty$, determine necessary and sufficient
geometric conditions on a convex boundary for
$\|M_\sigma\|_{L^p\to L^p}<\infty$. Its precise cap-sufficiency
proposal asks whether $\mathcal C_p(\Omega)<\infty$ implies this
bound for $p\ne2$.
\end{problem}

Our first theorem disproves this proposal throughout two open exponent
ranges, one below two and the other above three. These counterexamples
motivate sufficient conditions stated directly in geometric terms.
Our positive result uses one logarithmic curvature integral. A lower
envelope incorporates the local variation of curvature and allows the
criterion to be stated on a given convex boundary without a smoothness
assumption. The general geometric characterization
asked for in \Probref{prob:wright} remains open.

\subsection{Main results}\label{subsec:main-results}
We use the weak norm and the restricted weak-type constant
\[
 \begin{split}
 \|g\|_{L^{p,\infty}}
    &=\sup_{\lambda>0}\lambda\,
                          |\{x:|g(x)|>\lambda\}|^{1/p},\\
 \|T\|_{L^p\to L^{p,\infty}}^{\mathrm{restricted}}
    &=\sup_{\substack{E\subseteq\R^2\ \mathrm{measurable}\\0<|E|<\infty}}
        \frac{\|T\mathbf1_E\|_{L^{p,\infty}}}{|E|^{1/p}}.
 \end{split}
\]
We first show that finite cap energy can coexist with failure even on
characteristic functions. The examples below concern one fixed smooth
convex boundary for each stated exponent.
\begin{theorem}[Failure of universal cap sufficiency]
\label{Bneg:thm:general-negative}
For each fixed $p_0\in(1,(\sqrt{13}-1)/2)\cup(3,\infty)$ there is a bounded
$C^\infty$ strictly convex domain $\Omega_{p_0}$ containing the origin,
with exactly one curvature-zero point, such that
\begin{equation}\label{Bneg:eq:general-obstruction}
 \begin{gathered}
 \mathcal C_{p_0}(\Omega_{p_0})<\infty
 \quad\text{and}\\
 \|M_{\sigma_{p_0}}\|_{L^{p_0}\to L^{p_0,\infty}}^{\mathrm{restricted}}
 =\|M_{\sigma_{p_0}}\|_{L^{p_0}\to L^{p_0,\infty}}
 =\|M_{\sigma_{p_0}}\|_{L^{p_0}\to L^{p_0}}=\infty.
 \end{gathered}
\end{equation}
Here $\sigma_{p_0}$ is complete unnormalised true arclength, with
one fixed interior origin and all $k\in\Z$.
\end{theorem}

For the positive result, let $\mu$ be the turning measure of the
convex boundary and let $\kappa=d\mu_{\rm ac}/d\sigma$ be its a.e.
curvature. Corner angles and singular turning remain in $\mu$.
Write $d_{\partial\Omega}(x,y)$ for the shorter arclength distance along
the boundary, and define the curvature lower envelope
\begin{equation}\label{lc:eq:curvature-envelope}
 \kappa_*(x)=\operatorname*{ess\,inf}_{y\in\partial\Omega}
          \{\kappa(y)+d_{\partial\Omega}(x,y)\}.
\end{equation}
The essential infimum uses true arclength. This is the largest
$1$-Lipschitz function lying below $\kappa$ a.e. It controls local
curvature oscillation without imposing a regularity class.
\begin{theorem}[A single sufficient curvature condition]
\label{cond:thm:cap-iff}
Let $\Omega$ be any given bounded convex planar domain with its fixed
interior origin and complete unnormalised arclength $\sigma$.
For every finite $1<p<\infty$, the single condition
\begin{equation}\label{lc:eq:main-condition}
 \int_{\partial\Omega}
       [\log(e+\kappa_*(x)^{-1})]^{1/p}\dd\sigma(x)<\infty
\end{equation}
implies
\begin{equation}\label{cond:eq:cap-upper-bound}
 \|M_\sigma\|_{L^p(\R^2)\to L^p(\R^2)}<\infty.
\end{equation}
The integrand is $+\infty$ wherever $\kappa_*=0$.
\end{theorem}

\begin{remark}
For a $C^{2,\gamma}$ boundary, $0<\gamma\le1$, the curvature satisfies
$\kappa_*\gtrsim_{\Omega,\gamma}\kappa^{1/\gamma}$. Thus the familiar
condition $\int_{\partial\Omega}[\log(e+\kappa^{-1})]^{1/p}\dd\sigma<\infty$
is sufficient, as verified in \Secref{lc:sec:turning-examples}.
The published necessity \Eqref{eq:published-necessary} gives the original
cap condition from the conclusion. The logarithmic power $1/p$ cannot
be uniformly lowered, by \Secref{lc:sec:power}.
\end{remark}

\subsection{Organization}\label{subsec:organization}
\Secref{sec:main-proofs} proves \Thmref{Bneg:thm:general-negative}
and \Thmref{cond:thm:cap-iff}. \Secref{sec:auxiliary} supplies the
Fourier estimates and geometric constructions used in those proofs.

\subsection{Notation}\label{subsec:notation}
The relation $X\lesssim Y$ means $X\le CY$, and $X\eqsim Y$ means
that both comparison inequalities hold. Unless stated otherwise,
constants may depend on the exponent, the given domain and the smooth
cutoffs. They are independent of the function, finite
dilation families and measurable scale selectors.
The cap depth $\delta$ varies throughout the estimates.
For normed or quasi-normed spaces $X,Y$, we write
\[
       \|T\|_{X\to Y}=\sup_{f\ne0}\frac{\|Tf\|_Y}{\|f\|_X}.
\]
We abbreviate $\|T\|_{L^p(\R^2)\to L^p(\R^2)}$ by
$\|T\|_{p\to p}$. For a positive measure $\eta$ and finite $I\subseteq\Z$, set
$M_{\eta,I}f=\max_{k\in I}|\eta_{2^k}*f|$ and
$M_\eta f=\sup_{k\in\Z}|\eta_{2^k}*f|$.
$M_{\rm HL}$ denotes the Hardy--Littlewood maximal operator.
Reflection of a measure is denoted by
$\widetilde\eta$, the dilate $\eta_s$ is its pushforward under
$y\mapsto sy$, and $L^p(x,k)$ uses counting measure in $k$.
Our Fourier convention is
$\widehat\eta(\xi)=\int e^{-2\pi i y\cdot\xi}\dd\eta(y)$, and
$\log^+t=\max\{0,\log t\}$.

\section{Proofs of the main results}\label{sec:main-proofs}
\subsection{Proof of \texorpdfstring{\Thmref{Bneg:thm:general-negative}}{Theorem 1.1}}
\begin{proof}[Proof of \Thmref{Bneg:thm:general-negative} below two]
Fix $1<p=p_0<(\sqrt{13}-1)/2$ and $a=1$ in
\Secref{Bneg:subsec:negative-geometry}. Translate one interior point to zero
once, and retain the one fixed enlargement $L_0$. The constructed oval
is smooth and strictly convex, has exactly one curvature-zero point, and
satisfies $\mathcal C_{p_0}<\infty$.

Let $m_n$ be the unscaled hole radius. At each actual centre choose a
symmetric tangent-coordinate subarc with
\[
 \mu_n=L_0m_n/4\quad\text{and}\quad
 \mathfrak d_n\le C L_0\varepsilon_nm_n^2,\qquad
 \varepsilon_n=e^{-m_n^{-p_0}}.
\]
Its arclength density is at least $ds$, and its exact tangent is
$A_n^{-1}v_i/|A_n^{-1}v_i|$. The mass in
\Lemref{Bneg:lem:rational-lattice-scale-lower} is therefore $\mu_n$. Set
\[
 \begin{gathered}
 K_n^{\rm phys}=\left\lfloor\frac{m_n^{-p_0}}{2\log2}\right\rfloor,\qquad
 k_{0,n}=\max\left\{0,\left\lceil\log_2(C_1Q_n/\mu_n)\right\rceil\right\},\\
 J_n=K_n^{\rm phys}-k_{0,n}+1\quad\text{and}\quad r_n^{\rm in}=\frac{c_1}{N_nJ_nQ_n}.
 \end{gathered}
\]
The established growth estimates give
\begin{align*}
 2^{k_{0,n}}\mu_n&\ge C_1Q_n,\\
 \log M_n&=O(t_n^{-a}+n),\qquad
 \log N_n=O(t_n^{-a}/n^2),\\
 \log(1/m_n)&=O(\log N_n+n)=o(m_n^{-p_0}),\\
 k_{0,n}&=o(m_n^{-p_0})\quad\text{and}\quad J_n\eqsim m_n^{-p_0},\\
 2\|A_n\|2^{K_n^{\rm phys}}\mathfrak d_n
 &\le C L_0M_nm_n^2e^{-m_n^{-p_0}/2}
 \le r_n^{\rm in}/100.
\end{align*}
The last inequality follows since the logarithms of all prefactors and
of $1/r_n^{\rm in}$ are $o(m_n^{-p_0})$.

For the fixed inball radius $h_0>0$,
\Lemref{Bneg:cor:interior-origin-normal-phase} yields
\[
 |n_i\cdot A_nz_i|\ge h_0|A_n^Tn_i|\ge h_0M_nQ_n,
 \qquad
 2^{k_{0,n}}\min_i|n_i\cdot A_nz_i|\ge1.
\]
Thus all hypotheses of \Lemref{Bneg:lem:rational-lattice-scale-lower} hold at
the actual curve points. Apply it on the consecutive physical indices
$k_{0,n}\le k\le K_n^{\rm phys}$, and choose its scalar
$S_n\in[1,2]$ for each block.
For bounded truncations of the resulting single periodic indicator,
\begin{align*}
 \|M_\sigma\|_{L^{p_0}\to L^{p_0,\infty}}^{\mathrm{restricted}}
 &\ge c\mu_n(N_nJ_n)^{2/p_0-1}Q_n^{2/p_0-2}\\
 &\eqsim b_n^{p_0-1}N_n^{3/p_0-p_0-1}
 \longrightarrow\infty,
\end{align*}
because
\[
 m_n\eqsim b_n/N_n,\qquad Q_n\eqsim N_n^{1/2},\qquad
 J_n\eqsim m_n^{-p_0},\qquad
 |\log b_n|=O(n),\qquad \frac{\log N_n}{n}\longrightarrow\infty.
\]
The exponent $3/p_0-p_0-1$ is positive for precisely this range.
The logarithm of the lower bound is
$(3/p_0-p_0-1)\log N_n+O_{p_0}(n)\to\infty$.
The factor $Q_n$ is retained. Each bounded indicator supplies every
convolution in its finite family. On the union of the successful sets,
choose one successful arc and its physical index. The complete positive
average dominates that arc average. The domain and origin are fixed
throughout, and the full bilateral supremum contains every such finite
family.
\end{proof}

\subsection{The above-two case of \texorpdfstring{\Thmref{Bneg:thm:general-negative}}{Theorem 1.1}}
\begin{proof}[Proof of \Thmref{Bneg:thm:general-negative} above three]
Fix $p_*=p_0>3$ and $s_0=(p_*+3)/2$, so $3<s_0<p_*$. Apply \Secref{Bneg:sec:geometry} and \Secref{Bneg:sec:Ctest} with
$a=1/2$ and floor exponent $s_0$. They give one smooth strictly convex
domain with one curvature-zero point and an interior origin fixed before
phase programming. By \Lemref{Bneg:lem:C-cap} and
\Eqref{Bneg:Ctest:ten-thirds-divergence},
\[
 \begin{gathered}
 \mathcal C_{p_*}(\Omega)<\infty,\\
 \|M_\sigma\|_{L^{p_*}\to L^{p_*,\infty}}^{\mathrm{restricted}}=\infty.
 \end{gathered}
\]
\end{proof}

\subsection{Proof of the geometric sufficient condition}
We split the original measure where its curvature lower envelope is
small. The curvature pays the Fourier estimate, and the single integral
pays the remaining mass at every depth.
\begin{proof}[Proof of \Thmref{cond:thm:cap-iff}]
Write $P=\sigma(\partial\Omega)$ and use an arclength parameter $s$ on
$[0,P]$, with the endpoints identified only to describe the closed curve.
Set $p'=p/(p-1)$, $w=\kappa_*$ and
$W=[\log(e+w^{-1})]^{1/p}$. The triangle inequality
makes $w$ $1$-Lipschitz and $w\le2\pi/P+P/2$. Lebesgue differentiation
gives $w\le\kappa$ a.e., hence $\mu\ge w\dd s$.
The hypothesis is exactly $\int_0^P W(s)\dd s<\infty$.
Let $L_l$ have multiplier $\chi(2^{-l}\xi)$, where $\chi$ is smooth,
real and radial, one on the unit ball and zero outside radius two.
Write $P_l=L_l-L_{l-1}$ and set $\delta_j=2^{-j/4}$. Fix a smooth
nondecreasing $\psi:\R\to[0,1]$, zero on $(-\infty,1/2]$ and one on
$[1,\infty)$. The original positive measures are
\begin{equation}\label{lc:eq:actual-split}
 \nu_j^+=\psi(w/\delta_j)\sigma,\qquad
 \nu_j^-=\sigma-\nu_j^+,\qquad m_j=\nu_j^-(\R^2).
\end{equation}
Here $\nu_j^+$ increases, $\nu_j^-$ decreases and
$m_j\le\sigma\{w<\delta_j\}$. The finite logarithmic integral gives
$\sigma\{w=0\}=0$, hence $m_j\downarrow0$ by dominated convergence.
For $1<p\le2$, Minkowski's inequality in $\ell^p$ gives
\begin{equation}\label{lc:eq:lorentz-small}
 \begin{split}
 (\sum_jm_j^p)^{1/p}
 &\le\int_0^P\Big(\sum_j\mathbf1_{\{w(s)<\delta_j\}}\Big)^{1/p}\dd s\\
 &\lesssim_p\int_0^P W(s)\dd s.
 \end{split}
\end{equation}
For $p>2$, let $W^*$ be the decreasing
rearrangement of $W$. The complete groups
\[
 J_l=\{j:P2^{-l-1}<m_j\le P2^{-l}\}
 \quad\text{and}\quad K_l=\#J_l
\]
satisfy $K_l\lesssim[W^*(P2^{-l-1})]^p$. Therefore
\begin{equation}\label{lc:eq:lorentz-large}
 \begin{split}
 \Big[\sum_lK_l^{1/(p-1)}(P2^{-l})^{p'}\Big]^{1/p'}
 &\lesssim_p\sum_l P2^{-l}W^*(P2^{-l-1})\\
 &\le4\int_0^P W(s)\dd s.
 \end{split}
\end{equation}
The disjoint intervals $[P2^{-l-2},P2^{-l-1}]$ pay the last sum.
Let $\tau$ be the a.e. unit tangent and write
\begin{equation}\label{lc:eq:grazing-mass}
 G_j(e)=\int_{\{|e\cdot\tau(s)|\le16\delta_j^2\}}
                      \psi(w(s)/\delta_j)\dd s.
\end{equation}
For large $j$, actual angle integration gives
\begin{equation}\label{lc:eq:curvature-grazing}
 G_j(e)\le2\delta_j^{-1}
       \int_{\{|e\cdot\tau|\le16\delta_j^2\}}\kappa\dd s
       \lesssim_\Omega\delta_j^{-1}2^{-j/2}=2^{-j/4}.
\end{equation}
Indeed $w>\delta_j/2$ on the mask support and $w\le\kappa$ a.e.
Monotone turning bounds the curvature integral by the angular width.
Corner points have zero arclength, and singular turn only strengthens
$\mu\ge\kappa\dd s$. The perimeter pays the finite prefix.
Fix finite $I\subseteq\Z$ and first restrict to $1\le j\le N$.
Compactness of the original support gives the lowpass bound
\[
 \sup_{k\in I}|\sigma_{2^k}*L_{-k}f|
                         \lesssim_\Omega P M_{\rm HL}f.
\]
Suppose $1<p<2$. Set $A_I=\|M_{\sigma,I}\|_{p\to p}<\infty$ by
finite-family Young. Freeze disjoint measurable $E_k$, unit complex
phases $\omega_k$, and one function $h$. Write
\[
 h_k=\omega_kh\mathbf1_{E_k},\qquad
 U_{j,k}^\pm=\widetilde{\nu_j^\pm}_{2^k}*P_{j-k}h_k,
 \qquad a_j=\sup_eG_j(e)+C_\Omega\delta_j.
\]
Zero-mass rows are omitted. \Lemref{lc:lem:masked-fourier}, Plancherel
and radial Littlewood--Paley analysis give
\begin{equation}\label{lc:eq:parent-endpoints}
 \|\{m_j^{-1}U_{j,k}^-\}\|_{L^2(\ell^2)}\lesssim\|h\|_2
 \quad\text{and}\quad
 \|\{a_j^{-1}U_{j,k}^+\}\|_{L^2(\ell^2)}\lesssim\|h\|_2.
\end{equation}
For the other endpoint, Jensen and $\nu_j^-\le\sigma$, paired
with a nonnegative $v\in L^p$, yield
\[
 \begin{split}
 \sum_{j,k}\int m_j^{-1}|U_{j,k}^-|^2v
 &\le\sum_{j,k}\int(\nu_{j,2^k}^-*v)|P_{j-k}h_k|^2\\
 &\le\int M_{\sigma,I}v\sum_{j,k}|P_{j-k}h_k|^2
 \lesssim_p A_I\|v\|_p\|h\|_{2p'}^2.
 \end{split}
\]
Hilbert-valued radial analysis \cite{Stein} precedes disjointness of
$h_k$. The plus array has the same estimate with weight $P^{-1}$.
Interpolate with $L^2$ proportion $p-1$. The resulting weights are
$m_j^{-p/2}$ and $a_j^{-(p-1)}P^{-(2-p)/2}$. Weighted Cauchy--Schwarz
on $j-k=l$, followed by radial synthesis, gives
\begin{equation}\label{lc:eq:parent-synthesis}
 \begin{split}
 \|\sum_{j,k}U_{j,k}^-\|_{p'}
 &\lesssim_p(\sum_jm_j^p)^{1/2}A_I^{(2-p)/2}\|h\|_{p'},\\
 \|\sum_{j,k}U_{j,k}^+\|_{p'}
 &\lesssim_p(P^{2-p}\sum_ja_j^{2p-2})^{1/2}
                                      A_I^{(2-p)/2}\|h\|_{p'}.
 \end{split}
\end{equation}
Each summand retains its radial support. Since
$\sum_ja_j^{2p-2}<\infty$, duality,
\Eqref{lc:eq:lorentz-small} and the radial limit below imply
\begin{equation}\label{lc:eq:parent-absorption}
 \begin{split}
 A_I&\le C_{\Omega,p}\Big[1+
           \Big(\int_0^P W\dd s\Big)^{p/2}\Big]A_I^{(2-p)/2}
           +C_{\Omega,p}\\
 &\Longrightarrow\quad A_I<\infty
       \quad\text{uniformly in }I.
 \end{split}
\end{equation}
At $p=2$, first use $\nu_j^++\nu_j^-=\sigma$ at the same $j$.
On its annulus, $|\widehat\sigma(Re)|\le G_j(e)+C_\Omega\delta_j+m_j$.
For fixed $k,\xi$ only boundedly many $j$ contribute, and for fixed
$j,\xi$ only boundedly many $k$ contribute. The direction
$e=\xi/|\xi|$ remains fixed in the complete sum, so
\[
 \sum_{k\in I}\|\sum_{j=1}^N\sigma_{2^k}*P_{j-k}f\|_2^2
 \lesssim_\Omega[1+(\int_0^P W\dd s)^2]\|f\|_2^2.
\]
For $p>2$, \Lemref{lc:lem:nested-mass} and \Eqref{lc:eq:lorentz-large}
bound the selected complementary sum by
$C_p(\int_0^P W\dd s)\|f\|_p$.
For the increasing $\nu_j^+$, Abel summation as in
\Eqref{lc:eq:nested-abel} bounds every selected partial kernel by $CP$
in total variation. Its $L^2$ norm is bounded by
$C_\Omega[1+\sup_e\sum_jG_j(e)^2]^{1/2}$ by
\Eqref{lc:eq:same-ray-energy}. Scalar $2/\infty$ interpolation
\cite{BL} consequently gives a bound independent of $I,N$ and the
selector, by \Eqref{lc:eq:curvature-grazing}.
Finally the exact original identity is
\begin{equation}\label{lc:eq:radial-original-limit}
 \begin{split}
 \sum_{j=1}^N(\nu_j^++\nu_j^-)_{2^k}*P_{j-k}f
 &=\sigma_{2^k}*(L_{N-k}-L_{-k})f\\
 &\longrightarrow\sigma_{2^k}*(f-L_{-k}f)\quad\text{in }L^p.
 \end{split}
\end{equation}
The approximate identity and Young with the full mass $P$ give
simultaneous convergence for finite $I$. Thus the truncated bounds
apply to the original finite maximum before \Eqref{lc:eq:parent-absorption}.
Increasing $I$ to $\Z$ and applying Fatou proves the theorem.
\end{proof}

\section{Proofs of the auxiliary estimates}\label{sec:auxiliary}
\subsection{Fourier decay and nested positive measures}
The following estimate keeps the full directional energy on one ray.
\begin{lemma}\label{lc:lem:masked-fourier}
Write $\delta_j=2^{-j/4}$. The measures in \Eqref{lc:eq:actual-split} satisfy
\begin{equation}\label{lc:eq:masked-annular}
 \sup_{2^{j-2}\le R\le2^{j+2}}|\widehat{\nu_j^+}(Re)|
                          \le G_j(e)+C_\Omega2^{-j/4}.
\end{equation}
For finite $I\subseteq\Z$, $N\ge1$, disjoint measurable $E_k$ and
unit phases $\omega_k$, the same common function $f$ satisfies
\begin{equation}\label{lc:eq:same-ray-energy}
 \|\sum_{k\in I}\mathbf1_{E_k}\omega_k
       \sum_{j=1}^N\nu_{j,2^k}^+*P_{j-k}f\|_2^2
 \lesssim_\Omega[1+\sup_e\sum_jG_j(e)^2]\|f\|_2^2.
\end{equation}
\end{lemma}
\begin{proof}
Partition the actual boundary into finitely many small-turn arcs of
length at most $P/4$. Large corner jumps are endpoints. Each arc is a
small-slope convex Lipschitz graph $(x,\gamma(x))$, with monotone
$\gamma'$ and arclength density $a=(1+\gamma'^2)^{1/2}$ of bounded
supremum and variation. Since $w$ is $1$-Lipschitz,
\[
 \|a\psi(w/\delta_j)\|_\infty+
 \Var[a\psi(w/\delta_j)]\lesssim_\Omega 1+\delta_j^{-1}.
\]
For $\phi(x)=R(e_1x+e_2\gamma(x))$, $\phi'$ is monotone.
The set $|\phi'|\le\sqrt R$ has $|e\cdot\tau|\le R^{-1/2}$,
so its absolute integral is paid once by $G_j(e)$.
On at most two complementary intervals $\phi'$ is one-signed and
$|\phi'|\ge\sqrt R$. The substitution $u=\phi(x)$ and BV
integration, including the endpoints, give
\[
 |\int_{\{|\phi'|>\sqrt R\}} e^{-2\pi i\phi(x)}a(x)\psi(w(x)/\delta_j)\dd x|
 \lesssim_\Omega(1+\delta_j^{-1})R^{-1/2}
 \lesssim_\Omega2^{-j/4}.
\]
Positive-length boundary levels may be included in the grazing set.
Summing the disjoint chart interiors proves \Eqref{lc:eq:masked-annular}.
Let $p_l$ be the multiplier of $P_l$. Disjointness and Plancherel give
\[
 \begin{split}
 \|\sum_{k\in I}\mathbf1_{E_k}\omega_k
       \sum_{j=1}^N\nu_{j,2^k}^+*P_{j-k}f\|_2^2
 &\le\sum_{k\in I}\|\sum_{j=1}^N\nu_{j,2^k}^+*P_{j-k}f\|_2^2\\
 &\lesssim\int\sum_{j,k}
 |\widehat{\nu_j^+}(2^k\xi)p_{j-k}(\xi)\widehat f(\xi)|^2\dd\xi.
 \end{split}
\]
For fixed $k,\xi$, only boundedly many $j$ occur. For fixed $j,\xi$,
only boundedly many $k$ occur, with $2^k|\xi|$ in the $j$th annulus.
The direction $\xi/|\xi|$ stays fixed throughout the sum.
Equation \Eqref{lc:eq:masked-annular} proves \Eqref{lc:eq:same-ray-energy}.
\end{proof}

Above two, nesting controls the remaining mass with the same selector.
\begin{lemma}\label{lc:lem:nested-mass}
Let $p>2$ and $\eta_{j+1}\le\eta_j$ be positive measures with
$m_j=\eta_j(\R^2)\downarrow0$ and $m_j\le P$. Write
\[
 J_l=\{j:P2^{-l-1}<m_j\le P2^{-l}\}
 \quad\text{and}\quad K_l=\#J_l.
\]
If these counts are finite, then uniformly in finite $I\subseteq\Z$ and $N$,
\begin{equation}\label{lc:eq:nested-mass-bound}
 \|\max_{k\in I}|\sum_{j=1}^N(\eta_j)_{2^k}*P_{j-k}f|\|_p
 \lesssim_p
 (\sum_{l\ge0}K_l^{1/(p-1)}(P2^{-l})^{p'})^{1/p'}\|f\|_p.
\end{equation}
\end{lemma}
\begin{proof}
Consecutive projection sums have uniformly $L^1$-bounded kernels.
Abel summation gives the actual identity
\begin{equation}\label{lc:eq:nested-abel}
 \begin{split}
 \sum_{j=a}^b(\eta_j)_{2^k}*\check p_{j-k}
 &=(\eta_b)_{2^k}*\sum_{j=a}^b\check p_{j-k}\\
 &\quad+\sum_{j=a}^{b-1}(\eta_j-\eta_{j+1})_{2^k}
                                      *\sum_{h=a}^j\check p_{h-k}.
 \end{split}
\end{equation}
Positive differences give total variation at most
$C[m_b+\sum_{j=a}^{b-1}(m_j-m_{j+1})]=Cm_a$.
For increasing positive measures of mass at most $P$, the analogous
first-term and suffix-sum identity gives $CP$.
Assume the sum in \Eqref{lc:eq:nested-mass-bound} is finite and positive,
and omit empty groups and zero rows. Choose
\begin{equation}\label{lc:eq:analytic-weights}
 \begin{gathered}
 u_l=(K_lP2^{-l})^{1/(p-1)},\\
 \sum_lP2^{-l}u_l=\sum_lK_l(P2^{-l})^2u_l^{2-p}
               =\sum_lK_l^{1/(p-1)}(P2^{-l})^{p'}.
 \end{gathered}
\end{equation}
The groups use all $j$, independently of $N$. Freeze disjoint $E_k$
and unit phases $\omega_k$. For $0\le\Re z\le1$ set
\begin{equation}\label{lc:eq:analytic-family}
 T_zf=\sum_{k\in I}\mathbf1_{E_k}\omega_k\sum_lu_l^{1-pz/2}
       \sum_{\substack{j\in J_l\\j\le N}}(\eta_j)_{2^k}*P_{j-k}f.
\end{equation}
Each group is consecutive. Its coefficient on $\Re z=0$ has modulus
$u_l$ and is constant in $j$. Equation \Eqref{lc:eq:nested-abel} gives
the first bound below. On $\Re z=1$, Plancherel, disjoint selectors,
$|\widehat\eta_j|\le m_j$ and the two radial overlaps give the second:
\begin{equation}\label{lc:eq:analytic-endpoints}
 \begin{split}
 \|T_{it}\|_{\infty\to\infty}
 &\lesssim\sum_lK_l^{1/(p-1)}(P2^{-l})^{p'},\\
 \|T_{1+it}\|_{2\to2}^2
 &\lesssim\sum_lK_l^{1/(p-1)}(P2^{-l})^{p'}.
 \end{split}
\end{equation}
At fixed $k,\xi$ only boundedly many $j$ occur, then at fixed $j,\xi$
only boundedly many $k$ occur. These estimates are uniform in $t,I,N$
and the selector. At $z=2/p$ all coefficients equal one, so scalar
analytic interpolation \cite{BL} gives the $1/p'$ power of this sum.
Taking the supremum over the same selectors proves the lemma.
\end{proof}

\subsection{Comparison with the ordinary curvature integral}\label{lc:sec:turning-examples}
We verify that the lower envelope preserves the simple curvature
criterion on H\"older-continuous boundaries. Suppose
$\partial\Omega\in C^{2,\gamma}$, $0<\gamma\le1$. For each pair $x,y$,
either $\kappa(y)\ge\kappa(x)/2$, or the H\"older estimate gives
\[
 d_{\partial\Omega}(x,y)
 \ge\left(\frac{\kappa(x)}{2[\kappa]_{C^\gamma}}\right)^{1/\gamma}.
\]
Taking the essential infimum in \Eqref{lc:eq:curvature-envelope} yields
\[
 \kappa_*(x)\gtrsim_{\Omega,\gamma}\kappa(x)^{1/\gamma}
 \quad\text{and}\quad
 [\log(e+\kappa_*^{-1})]^{1/p}
       \lesssim_{\Omega,p,\gamma}[\log(e+\kappa^{-1})]^{1/p}.
\]
Constant curvature is immediate. Thus the ordinary curvature integral
implies \Eqref{lc:eq:main-condition}, with arbitrary curvature
oscillations and any number of curvature zeros.

\subsection{Uniform sharpness of the logarithmic power}\label{lc:sec:power}
Fix $p>1$ and prescribe the actual smooth $2\pi$-periodic curvature
\[
 \kappa(s)=c_p e^{-|\sin s|^{-p}},\qquad
 c_p=\frac{2\pi}{\int_0^{2\pi}e^{-|\sin v|^{-p}}\,dv},
 \qquad \kappa=0\ \text{on }\pi\Z.
\]
All jets vanish at the two zeros. Define
\[
 \theta(s)=\int_0^s\kappa(v)\,dv,\qquad
 \Gamma_0(s)=\int_0^s(\cos\theta(v),\sin\theta(v))\,dv.
\]
The curvature is $\pi$-periodic with half-period integral $\pi$, so
$|\Gamma_0'|=1$, $\Gamma_0(s+\pi)=\Gamma_0(\pi)-\Gamma_0(s)$ and
$\Gamma_0(2\pi)=0$. Its angle strictly increases on every interval.
At each tangent, the inward-normal height first strictly increases and
then strictly decreases to zero only at the full period. Hence the curve
is embedded and strictly convex. Translate its interior symmetry center
to zero once, fixing
one smooth domain with true perimeter $2\pi$ and complete $\sigma=ds$.

The preceding comparison and $\kappa_*\le\kappa$ give
$\log(e+\kappa_*^{-1})\eqsim\log(e+\kappa^{-1})$.
For every $0<d<1/p$,
\[
 [\log(e+\kappa_*^{-1})]^d
       \eqsim\dist(s,\pi\Z)^{-pd}\quad\text{near the zeros},
 \qquad
 \int_0^{2\pi}[\log(e+\kappa_*^{-1})]^d\dd s<\infty.
\]
At $s=0$ the actual height above its fixed supporting tangent satisfies
\[
 h(s)=\int_0^s\sin\theta(v)\,dv
 \le\int_0^s(s-v)\kappa(v)\,dv
 \le(c_p/2)s^2e^{-s^{-p}}=:D(s).
\]
Here $|\sin v|\le v\le s$ for small $s>0$, and $D$ is increasing.
The cap in this one fixed normal $e_0$ contains the initial arc of length
$s$ at depth $D(s)$. Choose $s_*$ with $D(s_*)<\delta_*$. Thus
\[
 \mathcal C_p(\Omega)\ge
 \int_0^{D(s_*)}\Lambda_\Omega(e_0,\delta)^p\frac{d\delta}{\delta}
 \ge\int_0^{s_*}s^p\left(\frac2s+\frac p{s^{p+1}}\right)ds=\infty.
\]
\Eqref{eq:published-necessary} excludes strong $L^p$ boundedness of this
same complete original maximal operator. Hence $1/p$ cannot be uniformly
replaced by any smaller positive power in \Eqref{lc:eq:main-condition}.
\subsection{The lattice lower bound on one fixed curve}
We average overlapping periodic intervals over one scalar parameter.
This controls intersections from distinct physical scales in the lattice
test.
\begin{lemma}[Scalar averaging of periodic interval intersections]
\label{Bneg:lem:scalar-periodic-overlap}
Let $0<\rho<1/4$, let $I_\rho\subseteq\mathbb T$ be the interval
of radius $\rho$, and set
\[
 H_\rho(t)=|I_\rho\cap(I_\rho+t)|.
\]
Then $H_\rho$ is nonnegative, one-periodic, and
$\int_0^1H_\rho(t)\,dt=4\rho^2$. For every real $D\ne0$,
\begin{equation}\label{Bneg:eq:scalar-periodic-overlap-average}
 \left|\int_1^2H_\rho(DS)\,dS-4\rho^2\right|
       \le\frac{4\rho^2}{|D|}.
\end{equation}
For any finite $I\subseteq\mathbb Z$ with $k_0=\min I$,
\begin{equation}\label{Bneg:eq:dyadic-reciprocal-pair-sum}
 \sum_{\substack{k,l\in I\\k>l}}
       \frac1{2^k-2^l}\le4\,2^{-k_0}.
\end{equation}
\end{lemma}

\begin{proof}[Proof of \Lemref{Bneg:lem:scalar-periodic-overlap}]
The formula for the mean of $H_\rho$ is the integral of a
circle convolution: it is $|I_\rho|^2=(2\rho)^2$.
Write $\mu=4\rho^2$. For any interval of length $L>0$, split
it, starting at its own left endpoint, into $\lfloor L\rfloor$
complete periods and one interval of length less than one.
The latter has integral between zero and $\mu$, since
$H_\rho\ge0$. Thus the integral over the full interval differs
from $L\mu$ by at most $\mu$. The change of variable $t=DS$
proves \Eqref{Bneg:eq:scalar-periodic-overlap-average}, including
negative $D$.
For the other assertion, enlarge the index set to all integers
at least $k_0$ and set $k=l+d$:
\[
 \sum_{l\ge k_0}2^{-l}\sum_{d\ge1}\frac1{2^d-1}
 \le 2^{1-k_0}\sum_{d\ge1}2^{1-d}
 =4\,2^{-k_0}.
\]
The sums converge absolutely.
\end{proof}

We now combine the overlap estimate with rational tangent directions.
The resulting characteristic function tests all selected arc averages
simultaneously.
\begin{lemma}[Rational stripes with a varying lattice]
\label{Bneg:lem:rational-lattice-scale-lower}
Fix $1<p<2$ and $Q\ge1$. Let $A$ be an invertible real $2\times2$ matrix,
and let $v_i=(a_i,b_i)\in\mathbb Z^2$, $1\le i\le N$, be
pairwise nonparallel primitive vectors with
$Q\le a_i,b_i\le2Q$. Set
\[
 n_i=(-b_i,a_i),\qquad
 \tau_i=\frac{A^{-1}v_i}{|A^{-1}v_i|},\qquad
 \beta_i=n_i\cdot Az_i.
\]
Assume uniform constants $0<c<C$ satisfy
\[
 cQ\le|A^{-1}v_i|\le CQ\quad\text{and}\quad \|A\|\le H.
\]
Let one fixed positive compactly supported measure $\sigma$
contain true-arclength submeasures parametrized by
\[
 \begin{gathered}
 y_i(s)=z_i+s\tau_i+e_i(s),\qquad s\in[-\ell_i/2,\ell_i/2],\\
 m\le\ell_i\le Cm,\qquad |e_i(s)|\le d,\qquad d\nu_i\ge ds.
 \end{gathered}
\]
Let $I\subseteq\mathbb Z$ be finite, $J=|I|\ge1$,
$k_0=\min I$, and $K=\max I$. Suppose
\begin{equation}\label{Bneg:eq:lattice-scale-conditions}
 2^{k_0}m\ge C_1Q,\qquad
 2H2^Kd\le r/100,\qquad
 2^{k_0}\min_i|\beta_i|\ge1,\qquad
 r=\frac{c_1}{NJQ},
\end{equation}
where $C_1$ is sufficiently large and $c_1>0$ sufficiently
small, in terms only of the fixed constants above.
Then the finite maximal operator of this unchanged $\sigma$
satisfies the restricted weak lower bound
\begin{equation}\label{Bneg:eq:lattice-scale-weak-lower}
 \left\|\sup_{k\in I}|\sigma_{2^k}*\cdot|
       \right\|_{L^p\to L^{p,\infty}}^{\mathrm{restricted}}
 \ge c_p m(NJ)^{2/p-1}Q^{2/p-2}.
\end{equation}
More specifically, there exists $S\in[1,2]$ such that the
indicators witnessing the bound are bounded truncations of
\[
 f_S(x)=\mathbf1_{\{\operatorname{dist}(SAx,\mathbb Z^2)<r\}}.
\]
Different finite families
in the same curve may choose different $S$.
\end{lemma}

\begin{proof}[Proof of \Lemref{Bneg:lem:rational-lattice-scale-lower}]
Write $A_S=SA$ for $1\le S\le2$. Its inverse sends $v_i$
to the same tangent direction $\tau_i$, with physical period
$|A_S^{-1}v_i|=|A^{-1}v_i|/S\eqsim Q$ on the transformed
torus. The transformed speed along the unit tangent is
\[
 |A_S\tau_i|=\frac{S|v_i|}{|A^{-1}v_i|}\eqsim1,
\]
uniformly for all $S\in[1,2]$. Primitivity implies that if
\begin{equation}\label{Bneg:eq:lattice-scale-stripe}
 \operatorname{dist}\bigl(n_i\cdot A_Sx
       -2^kS\beta_i,\mathbb Z\bigr)<\rho,
 \qquad \rho=c_3Qr,
\end{equation}
with $c_3$ sufficiently small, the corresponding infinite
straight line passes within transformed distance $r/10$ of
an integer point during each complete period. It therefore
spends physical length at least $c r$ inside the transformed
radius-$r/2$ balls during each period. The line segment of
length $2^k\ell_i\ge C_1Q$ contains a fixed positive fraction
of complete periods. The true arc differs from that segment
in transformed coordinates by at most
$2H2^Kd\le r/100$. Positivity and $d\nu_i\ge ds$ consequently
give
\begin{equation}\label{Bneg:eq:lattice-scale-true-average}
 (\nu_i)_{2^k}*f_S(x)\ge c\frac{mr}{Q}
\end{equation}
on \Eqref{Bneg:eq:lattice-scale-stripe}. This is an estimate for
the original curved arclength, at the original physical scale.

For $S\in[1,2]$, define
\[
 U_i(S)=\bigcup_{k\in I}(2^kS\beta_i+I_\rho)\subseteq\mathbb T,
 \qquad u_i(S)=|U_i(S)|.
\]
By \Lemref{Bneg:lem:scalar-periodic-overlap}, Bonferroni, and
\Eqref{Bneg:eq:lattice-scale-conditions},
\begin{align*}
 \int_1^2u_i(S)\,dS
 &\ge2J\rho-4\rho^2\binom J2
       -\frac{4\rho^2}{|\beta_i|}
           \sum_{k>l}\frac1{2^k-2^l}\\
 &\ge2J\rho-2J(J-1)\rho^2
       -\frac{16\rho^2}{|\beta_i|2^{k_0}}
 \ge J\rho,
 \qquad J\rho=\frac{c_3c_1}{N}.
\end{align*}
Choose $c_1$ sufficiently small. Since $u_i(S)\le2J\rho$, at least one
$S\in[1,2]$ satisfies
\[
 c_3c_1=NJ\rho\le\sum_i u_i(S)\le2NJ\rho=2c_3c_1.
\]

Now work only in transformed coordinates $u\in\mathbb T^2$,
with that fixed $S$. Define
\[
 B_i=\{u:n_i\cdot u\in U_i(S)\}.
\]
The integer normals are nonparallel, so the torus homomorphism
$u\mapsto(n_i\cdot u,n_j\cdot u)$ is surjective for $i\ne j$.
It pushes Haar measure to Haar measure. Consequently
\[
 |B_i|=u_i(S)\quad\text{and}\quad |B_i\cap B_j|=u_i(S)u_j(S).
\]
Cauchy--Schwarz therefore gives
\[
 \begin{aligned}
 \left|\bigcup_iB_i\right|
 &\ge\frac{\left(\sum_i|B_i|\right)^2}
          {\sum_i|B_i|+2\sum_{i<j}|B_i\cap B_j|}\\
 &\ge\frac{\left(\sum_i u_i(S)\right)^2}
          {\sum_i u_i(S)+\left(\sum_i u_i(S)\right)^2}
 \ge c_4>0.
 \end{aligned}
\]
Only the integer normals $n_i$ enter this torus calculation. The matrices
$A$ and $A_S$ may have arbitrary real entries.

The periodic characteristic function $f_S$ has density $\pi r^2$ on the physical
lattice $A_S^{-1}\mathbb Z^2$. On a periodic set of
density at least $c_4$, one actual average in
\Eqref{Bneg:eq:lattice-scale-true-average} applies, and hence the
full positive $\sigma$ average at that physical index is no
smaller. Thus the original common physical supremum is at
least $c mr/Q$ there.

To pass to bounded indicators on $\mathbb R^2$, use
$f_{S,R}=f_S\mathbf1_{[-R,R]^2}$. Since $I$ is finite and
$\sigma$ is compactly supported, all original averages agree
with the periodic ones on $[-R+D_I,R-D_I]^2$, for some finite
$D_I$ independent of $R$. For any fixed invertible real
$A_S$, tile space by its bounded fundamental parallelograms.
Only $O(R)$ cells meet the boundary of the expanding square,
whereas its area is of order $R^2$. Normalized integrals of
bounded lattice-periodic functions therefore converge to
their fundamental-cell means. This applies both to the characteristic
function and to the union of the sets. The common cell Jacobian
$|\det A_S|^{-1}$ cancels. At level $\lambda=c mr/(2Q)$,
letting $R\to\infty$ gives the restricted weak ratio
\[
 c_p\frac{mr}{Q}(\pi r^2)^{-1/p}
 =c_p m(NJ)^{2/p-1}Q^{2/p-2}.
\]
The curve, its measure, and the physical supremum are unchanged
throughout. Only the test function and its lattice vary.
\end{proof}

The fixed interior origin keeps every supporting-line phase away from
zero. This supplies the phase hypothesis of the preceding lattice lemma
without moving the curve.
\begin{lemma}[Nonzero normal phase from an interior origin]
\label{Bneg:cor:interior-origin-normal-phase}
Suppose $\sigma$ is true arclength on the boundary of a fixed
convex domain containing the origin in its interior, with
support distances bounded below by $h_0>0$. Let $z_i$ be the
actual curve point of the $i$th arc, and assume that $\tau_i$
is its actual tangent at $z_i$. Then
\[
 |\beta_i|=|(A^Tn_i)\cdot z_i|
       \ge h_0|A^Tn_i|.
\]
For the matrices used in the rational constructions,
\[
 A=\begin{pmatrix}1&0\\-P&M\end{pmatrix},\qquad
 \theta_i=\frac{P+b_i/a_i}{M},\qquad M>0,
\]
one has
\[
 A^Tn_i=(-b_i-Pa_i,Ma_i),\qquad
 \beta_i=a_iM(y_i-\theta_i x_i),\qquad
 |\beta_i|\ge h_0MQ.
\]
Thus $2^{k_0}\min_i|\beta_i|\ge1$ automatically for the
large families and positive starting indices used below. The diagonal case is $P=0$.
\end{lemma}

\begin{proof}[Proof of \Lemref{Bneg:cor:interior-origin-normal-phase}]
The vector $A^Tn_i$ is perpendicular to $\tau_i$, since
$n_i\cdot v_i=0$. Its positive or negative unit normalization
is the supporting normal at $z_i$. The origin's positive
support distance gives the first assertion. The displayed
matrix calculation gives the rest. Tangent orthogonal
projection as the arc parameter ensures the required
$d\nu_i\ge ds$ in the preceding lemma.
\end{proof}

\subsection{A smooth convex graph with uniform cap energy}\label{Bneg:subsec:negative-geometry}
\subsubsection{A smooth convex graph with mass-preserving rational blocks}
\label{Bneg:subsub:rational-blocks}

Fix $1<p<(\sqrt{13}-1)/2$ and choose $0<a<p$.
All constants below are fixed before the block index tends to infinity.
Let
\[
 K_0(t)=e^{-|t|^{-a}}\quad(t\ne0),\qquad K_0(0)=0,
 \qquad \gamma_0(t)=\int_0^t(t-u)K_0(u)\,du
\]
with the usual oriented integral for negative $t$.
Choose a small fixed graph interval $(-2t_*,2t_*)$ so that
$|\gamma_0'|<1/100$ there.
Set $t_n=2^{-n}$ and $b_n=c\,t_n^{a+1}$ with a small fixed $c>0$.
Write
\[
 B_n=(l_n,r_n)=(t_n-b_n,t_n+b_n),\qquad
 D_n=\int_{B_n}K_0(t)\,dt.
\]
Start the construction at a sufficiently large index that all
blocks lie in $(0,t_*/4)$. For all such $n$, these blocks are disjoint,
$D_n\eqsim b_nK_0(t_n)$, and $K_0(t)\eqsim K_0(t_n)$ on $B_n$.
All comparability constants are independent of $n$.

\smallskip\noindent\textbf{Rational directions and their exact target slopes.}\space
Choose
\[
 Q_n=\left\lfloor\exp\left(\frac{t_n^{-a}}{4n^2}\right)\right\rfloor,
 \qquad
 \mathcal V_n=\{(u,v)\in\mathbb Z^2:Q_n\le u,v\le2Q_n,\ (u,v)=1\}.
\]
Let $N_n=|\mathcal V_n|$ and list these vectors in increasing order
of their ratio $v/u$. Then
\[
 c_1Q_n^2\le N_n\le C_1Q_n^2,\qquad
 \left|\frac vu-\frac{v'}{u'}\right|\ge\frac1{4Q_n^2}
 \quad\hbox{for distinct vectors}.
\]
For completeness, the lower counting bound follows by excluding
pairs divisible by $d\ge2$:
their number is at most
$Q_n^2\sum_{d\ge2}d^{-2}+O(Q_n\log Q_n)$,
strictly less than a fixed fraction of $Q_n^2$ for large $Q_n$.
Distinct positive primitive vectors cannot be parallel.

Define
\[
 M_n=\lceil20/D_n\rceil\quad\text{and}\quad \theta_n^-=\gamma_0'(l_n).
\]
Choose an integer $P_n$ such that every target
\[
 s_{n,i}=\frac{P_n+v_i/u_i}{M_n},\qquad (u_i,v_i)\in\mathcal V_n,
\]
lies in
$[\theta_n^-+D_n/4,\theta_n^-+3D_n/4]$.
Indeed the permissible interval for $P_n$ has length
$M_nD_n/2-3/2>1$.
For a fixed $c_2>0$, independent of $n$,
\begin{equation}\label{Bneg:eq:rational-target-increments}
 s_{n,i+1}-s_{n,i}\ge c_2D_n/N_n.
\end{equation}
The small graph slope gives $0<s_{n,i}<1/50$ after shortening
the chart. Also $M_nD_n\eqsim1$ and
\begin{equation}\label{Bneg:eq:matrix-growth}
 \log M_n=O(t_n^{-a}+\log(1/t_n)).
\end{equation}
The integer matrix
\[
 A_n=\begin{pmatrix}1&0\\-P_n&M_n\end{pmatrix}
\]
is invertible, and
\[
 A_n^{-1}(u_i,v_i)=(u_i,s_{n,i}u_i),\qquad
 |A_n^{-1}(u_i,v_i)|\eqsim Q_n,\qquad
 \|A_n\|\lesssim M_n.
\]
For the last inequality use $P_n/M_n\le1/50$.

\smallskip\noindent\textbf{Hole packing parameters.}\space
Let
\[
 d_n=\frac{2b_n}{N_n+1}\quad\text{and}\quad
 x_{n,i}=l_n+i d_n,\qquad
 m_n=\beta d_n,\qquad
 \varepsilon_n=e^{-m_n^{-p}},
\]
where the fixed $\beta>0$ will be chosen very small below.
The holes are
\[
 H_{n,i}=(x_{n,i}-m_n,x_{n,i}+m_n).
\]
They have equal parameter radius within a block but the radii
decrease supergeometrically between blocks:
$m_{n+1}/m_n\to0$.
In particular, after increasing the initial index,
\begin{equation}\label{Bneg:eq:small-radius-tail}
 \sum_{n\ge n_1}m_n\le2m_{n_1}
 \quad\hbox{for every }n_1.
\end{equation}
Also
\[
 \log N_n=o(t_n^{-a}),\qquad
 \varepsilon_n\ll K_0(t_n).
\]

\smallskip\noindent\textbf{Curvature inside each block.}\space
We construct a positive $C^\infty$ function $K_n$ on $B_n$.
Fix a small $c_0>0$, depending only on the constants in
\Eqref{Bneg:eq:rational-target-increments} and the comparison of
$D_n$ with $b_nK_0(t_n)$.
First take a smooth background $q_n^0$ which is $K_0$ on
endpoint collars of width $d_n/10$, is $c_0K_0$ away from
collars of width $d_n/5$, and lies between these values.
Thus $q_n^0=c_0K_0$ throughout every center-to-center interval.
Use a fixed bump $\psi$ with $0\le\psi\le1$, supported in
$(-1,1)$ and equal to one on $[-1/2,1/2]$, to set
\[
 q_n(t)=q_n^0(t)+\sum_i
 \psi((t-x_{n,i})/m_n)\,[\varepsilon_n-q_n^0(t)].
\]
For large $n$, $q_n\ge\varepsilon_n$ on the whole block,
it equals $\varepsilon_n$ on each central half-hole,
and outside the holes it is at least $c_0K_0$.

Let the partition nodes be $l_n,x_{n,1},\ldots,x_{n,N_n},r_n$,
and prescribe derivative values at them to be
\[
 \theta_n^-,\quad s_{n,1},\ldots,s_{n,N_n},\quad
                  \theta_n^-+D_n.
\]
On each consecutive node interval add a nonnegative smooth bump,
supported in its middle part and away from holes and endpoint
collars, whose integral equals the prescribed derivative increment
minus the integral of $q_n$ over that interval.
These residuals are positive: on center-to-center intervals,
$\int q_n\le Cc_0D_n/N_n$, whereas the target increment is at
least $c_2D_n/N_n$. Choose $c_0$ sufficiently small.
On the two endpoint intervals, $\int q_n=O(D_n/N_n)$, while
the target increments are at least $D_n/4$.
Thus these residuals are positive for large $n$ as well.
The gap between holes has a fixed positive fraction of $d_n$,
so all bumps can be supported as asserted.

Call the resulting function $K_n$. It equals $K_0$ on fixed
open neighborhoods of each endpoint and satisfies
\begin{equation}\label{Bneg:eq:mass-preserving-curvature}
 \int_{B_n}K_n=D_n\quad\text{and}\quad
 \int_{l_n}^{x_{n,i}}K_n=s_{n,i}-\theta_n^-.
\end{equation}
It remains $\varepsilon_n$ on every central half-hole and
satisfies $K_n\ge c_0K_0$ outside the holes. Bumps may have
large ratios to the baseline, but for every derivative order $r$
\begin{equation}\label{Bneg:eq:curvature-derivative}
 \|K_n^{(r)}\|_\infty
       \le C_rK_0(t_n)N_n^{r+1}b_n^{-r}.
\end{equation}
Indeed a bump has integral at most $D_n$, support scale
comparable to $d_n$, and hence derivative size
$O_r(D_nd_n^{-r-1})$.
Background and hole-cutoff derivatives obey a smaller bound.
The hole and gap supports have uniformly bounded pointwise
overlap, so the number of bumps does not add an extra factor.

\smallskip\noindent\textbf{One fixed smooth graph and its actual arcs.}\space
Define $K=K_n$ on $B_n$ and $K=K_0$ elsewhere.
Endpoint matching gives smoothness away from zero.
By \Eqref{Bneg:eq:curvature-derivative} and $\log N_n=o(t_n^{-a})$,
for every fixed $r,L$,
$\|K_n^{(r)}\|_\infty=o(t_n^L)$.
Thus $K$ extends $C^\infty$ to zero with every derivative zero,
and is positive elsewhere. Let $\gamma''=K$,
$\gamma(0)=\gamma'(0)=0$.
The integral identity \Eqref{Bneg:eq:mass-preserving-curvature} implies
\[
 \gamma'(t)=\gamma_0'(t)\quad(t\notin\bigcup_n B_n),\qquad
 \gamma'(x_{n,i})=s_{n,i}.
\]
These identities preserve curvature mass and slopes. They do not preserve the height first moment, and no equality of $\gamma$ with $\gamma_0$ is used. There is no recursive unknown incoming slope. The graph is
$C^\infty$, strictly convex, and has only one curvature-zero
point. Here $K=\gamma''$ is graph second derivative, whereas
geometric curvature is $K/(1+\gamma'^2)^{3/2}$. The two are
uniformly comparable on the reduced chart. Its slope stays
uniformly small because the integral of graph second derivative
on every block is unchanged.

On the central half-hole the graph is exactly quadratic:
\[
 \gamma(x_{n,i}+h)=\gamma(x_{n,i})+s_{n,i}h
                         +\tfrac12\varepsilon_nh^2
 \quad(|h|\le m_n/2).
\]
For $s_{n,i}=\gamma'(x_{n,i})$ the orthogonal tangent coordinate is
\[
 s=\sqrt{1+s_{n,i}^2}\,h+
       \frac{s_{n,i}\varepsilon_n}{2\sqrt{1+s_{n,i}^2}}h^2,
\]
and its derivative with respect to $h$ is bounded above and below.
In orthogonal tangent--normal coordinates the normal defect is
$\varepsilon_n h^2/(2\sqrt{1+s_{n,i}^2})$.
A symmetric subarc in $s$ therefore has parameter length
comparable to $m_n$, true arclength
$\sqrt{1+|\partial_s e|^2}\,ds\ge ds$, and exact tangent
$\tau_{n,i}=A_n^{-1}v_i/|A_n^{-1}v_i|$, and normal defect
bounded by $C\varepsilon_nm_n^2$.
Thus these are actual arcs suitable for the rational-periodic
lemma, applied to the actual curved subarcs.

\subsubsection{All supporting normals: uniform cap energy}\label{Bneg:subsec:cap-calculation}
Fix a large constant $A_*>0$ and then choose
$\beta<1/[100(A_*+1)]$.
The expanded holes
\[
 H_{n,i}^*=(x_{n,i}-(A_*+1)m_n,\,
                   x_{n,i}+(A_*+1)m_n)
\]
are pairwise disjoint, within and between blocks.
On every such expanded hole, $K\ge\varepsilon_n$:
outside its actual hole the curvature is at least $c_0K_0$,
which exceeds $\varepsilon_n$ for large $n$, and all added bumps
are nonnegative.

For a parameter interval $J$ of length $L$, the hole lengths
obey the packing estimate
\begin{equation}\label{Bneg:eq:hole-packing}
 \sum_{i:H_{n,i}\cap J\ne\varnothing}|H_{n,i}|
       \le 2\beta|J\cap B_n|+C m_n.
\end{equation}
This follows by counting equally spaced centers. At most a
fixed number of boundary holes are added. If every hole
intersecting $J$ has radius below $L/A_*$, sum
\Eqref{Bneg:eq:hole-packing} only over blocks containing at least
one hole meeting $J$, use their disjointness and
\Eqref{Bneg:eq:small-radius-tail}, and obtain
\[
 |\text{all holes meeting }J|
       \le 2\beta L+C'\,L/A_*<L/8
\]
by choosing $A_*$ large and $\beta$ small in that order.
Here the sum of boundary radii is controlled by the largest
radius among the intersecting blocks, irrespective of their
large numbers of holes.

For any supporting parameter $x$, side $e=\pm1$, and admissible
parameter distance $L$, set
\[
 D_{x,e}(L)=\gamma(x+eL)-\gamma(x)-eL\gamma'(x)
          =\int_0^L(L-v)K(x+ev)\,dv.
\]
If the segment meets a hole with $m_n\ge L/A_*$, it lies in its
expanded hole. Hence
\[
 D_{x,e}(L)\ge\tfrac12\varepsilon_nL^2.
\]
Moreover $x$ lies in that expanded hole. Their disjointness
ensures that for this fixed supporting parameter there is at
most one exceptional pair $(n,i)$, even as $L$ and $e$ vary.
Otherwise discard the holes, of total length below $L/8$.
In $0\le v\le3L/4$, also discard $|x+ev|<L/8$, which removes
length at most $L/4$. A set of length at least $3L/8$ remains,
where $L-v\ge L/4$ and
$K(x+ev)\ge c_0e^{-8^aL^{-a}}$. Consequently
\[
 D_{x,e}(L)\ge c L^2e^{-8^aL^{-a}}.
\]
Thus, uniformly in $x,e,L$,
\[
 \log^+\frac{C}{D_{x,e}(L)}
 \le C_1+C_2\log^+(1/L)+C_3L^{-a}
       +\mathbf1_{\{L\le A_*m_n\}}m_n^{-p},
\]
with the last term present for at most one exceptional hole
depending on $x$.

If $L_{x,e}(\delta)$ is the one-sided vertical-cap parameter
length, Tonelli gives
\[
 \int_0^{\delta_0}L_{x,e}(\delta)^p\frac{d\delta}{\delta}
 =p\int_0^{L_{\max}}L^{p-1}
                \left[\log\frac{\delta_0}{D_{x,e}(L)}\right]_+dL.
\]
Since $a<p$,
\begin{align*}
 \int_0^{L_{\max}}L^{p-1}
    \bigl[1+\log^+(1/L)+L^{-a}\bigr]\,\dd L&<\infty,\\
 m_n^{-p}\int_0^{A_*m_n}pL^{p-1}\,dL&=A_*^p.
\end{align*}
For each fixed contact there is at most one exceptional hole. Bounded
graph slope compares actual arclength with parameter length and normal
depth with vertical depth. Therefore, summing both sides at that contact,
\[
 \sup_x\int_0^{\delta_0}
    \bigl[L_{x,+1}(\delta)+L_{x,-1}(\delta)\bigr]^p
          \frac{d\delta}{\delta}<\infty.
\]
The supremum is taken after the full depth integral.

\subsubsection{Completion, fixed origin, and fixed rescaling}
The graph endpoints lie in compact positive-curvature regions.
Preserve their germs, parametrize short extensions by tangent
angle, and prescribe positive smooth radius of curvature on
the complementary angle interval, of length greater than $\pi$.
Choose three nonnegative smooth bumps with tangent-direction
integrals whose convex hull contains zero in its interior.
Positive linear combinations of these integrals span $\R^2$.
Adding suitable nonnegative multiples to the radius therefore
cancels the total displacement and closes the curve. Positive
radius on the added part and monotone turning give a simple
strictly convex oval, smoothly matching the original graph.

Use the original graph on $[-t_*,t_*]$ for the completion.
The central interval $[-t_*/2,t_*/2]$ contains all blocks.
By strict convexity and compactness, sufficiently small caps
centered in it remain in this larger graph interval.
Small caps centered outside the central interval avoid
$[-t_*/4,t_*/4]$, which contains all holes, and lie in a fixed compact
positive-curvature region. Such caps have uniformly finite
depth integrals. Intermediate depths have bounded logarithmic
length. Hence the complete oval has finite original uniform
$p$ cap functional.

Finally translate an interior point to zero and apply one
fixed dyadic isotropic enlargement large enough to contain
$B(0,10)$. This fixes one given domain and one interior origin.
Cap finiteness is preserved: if the enlargement factor is $L$,
then $\Lambda_{L\Omega}(\theta,\delta)
=L\Lambda_\Omega(\theta,\delta/L)$, and translation does not
alter cap masses. The rational tangent slopes and integer
matrices are unchanged by isotropic enlargement.
Arc masses and defects acquire fixed factors only, so the
rational-periodic lemma keeps $J\eqsim m_n^{-p}$ and the same
unbounded-power lower bound. The logarithmic matrix condition
follows from \Eqref{Bneg:eq:matrix-growth} and
$m_n^{-p}\gg t_n^{-a}$.
This completes the geometric construction and the cap estimate for the complete oval.

\subsection{The above-two graph and its directional cap integral}
\label{Bneg:sec:geometry}
Throughout the above-two construction fix
\[
 p_*>3,\qquad s_0=\frac{p_*+3}{2},\qquad a=\tfrac12.
\]
The letter $K$ in this construction denotes the actual graph second derivative.

\subsubsection{Rational tangents and positive movable connectors}
Repeat the geometric setup of \Secref{Bneg:subsub:rational-blocks}
with $a=1/2$. Use its baseline $K_0$, double integral $\gamma_0$,
blocks $B_n=(l_n,r_n)$, masses $D_n$, primitive vectors $v_{n,i}=(u_i,v_i)$,
and parameters $t_n,b_n,Q_n,N_n,M_n,P_n$. This is a separate graph,
and all symbols now refer to this choice of $a$. The setup gives
$D_n\eqsim b_ne^{-t_n^{-a}}$, $N_n\eqsim Q_n^2$ and
\begin{equation}\label{Bneg:eq:C-slopes}
 s_{n,i}=\frac{P_n+v_i/u_i}{M_n}
 \in[\gamma_0'(l_n)+D_n/4,\gamma_0'(l_n)+3D_n/4].
\end{equation}
In particular, $s_{n,i+1}-s_{n,i}\ge c_1D_n/N_n$ and
$0<s_{n,i}<1/50$. With $\theta_n^-=\gamma_0'(l_n)$,
\begin{equation}\label{Bneg:eq:C-matrix}
 \begin{gathered}
 A_n=\begin{pmatrix}1&0\\-P_n&M_n\end{pmatrix},\qquad
 A_n^{-1}v_{n,i}=(u_i,s_{n,i}u_i),\\
 |A_n^{-1}v_{n,i}|\eqsim Q_n,\quad
 \|A_n\|\lesssim M_n,\quad \log M_n=O(t_n^{-a}+n).
 \end{gathered}
\end{equation}
These are the geometric verifications from the previous construction.
Its restriction on $p$ concerns only the earlier lower bound and is
not used here. The new curvature floor is specified next.

Let
\begin{equation}\label{Bneg:eq:C-holes}
 \begin{split}
 d_n&=\frac{2b_n}{N_n+1},\quad x_{n,i}=l_n+i d_n,\quad m_n=\beta d_n,\\
 F_n&=m_n^{-s_0},\quad \varepsilon_n=e^{-F_n},\quad
 H_{n,i}=(x_{n,i}-m_n,x_{n,i}+m_n).
 \end{split}
\end{equation}
The constant $\beta>0$ is fixed sufficiently small below. The radii
satisfy $m_{n+1}/m_n\to0$. After increasing $n_0$,
\begin{equation}\label{Bneg:eq:C-radius-tail}
 \sum_{n\ge n_1}m_n\le2m_{n_1},\qquad n_1\ge n_0.
\end{equation}
Also $\log N_n=o(t_n^{-a})$ and $\varepsilon_n\ll K_0(t_n)$.

Use the smooth background and hole cutoffs from
\Secref{Bneg:subsub:rational-blocks}, with the new
$\varepsilon_n$ in \Eqref{Bneg:eq:C-holes}. Thus $q_n$ equals
$K_0$ near block endpoints, $\varepsilon_n$ on central half-holes,
and is at least $c_0K_0$ outside the holes. Prescribe the same
node slopes $\theta_n^-,s_{n,1},\ldots,s_{n,N_n},\theta_n^-+D_n$.
The residual mass on a node interval is its required slope increment
minus $\int q_n$. The slope separation and sufficiently small $c_0$
give, exactly as before,
\begin{equation}\label{Bneg:eq:C-residual}
 c_2D_n/N_n\le w_{n,i}\le D_n.
\end{equation}
This includes both endpoint intervals, whose slope increments are
at least $D_n/4$.

Choose once an even nonnegative smooth bump $\varphi$ supported
in $(-1,1)$, of integral one, and a small fixed $\eta>0$.
On each node interval use a connector of the form
\begin{equation}\label{Bneg:eq:C-movable-bump}
 \frac{w_{n,i}}{\eta d_n}
       \varphi\left(\frac{t-c_{n,i}}{\eta d_n}\right).
\end{equation}
Its centre $c_{n,i}$ may vary in a fixed closed middle interval
of length $c_3d_n>0$. Choose the allowed intervals and $\eta$
so that all supports stay a fixed fraction of $d_n$ away from
the holes and endpoint collars. This is possible for small
$\beta$, uniformly on all node intervals. Every such connector
has the prescribed mass $w_{n,i}$ and first moment $w_{n,i}c_{n,i}$.
It is the first moment of this connector which will be varied.

Let $K_n$ be the background plus these connectors. Every allowed
choice satisfies
\begin{equation}\label{Bneg:eq:C-masses}
 \int_{B_n}K_n=D_n\quad\text{and}\quad
 \int_{l_n}^{x_{n,i}}K_n=s_{n,i}-\theta_n^-.
\end{equation}
It agrees with $K_0$ on endpoint neighbourhoods, is exactly
$\varepsilon_n$ on every central half-hole, and obeys
$K_n\ge c_0K_0$ outside the holes. Uniformly in all centres,
\begin{equation}\label{Bneg:eq:C-jets}
 \|K_n^{(r)}\|_\infty
 \le C_rK_0(t_n)N_n^{r+1}b_n^{-r}
 =o(t_n^L)\quad\text{for every fixed }r,L.
\end{equation}
The derivative calculation in \Eqref{Bneg:eq:curvature-derivative}
is uniform in translations of each connector, and
$\log N_n=o(t_n^{-a})$ gives the last equality.

Define $K=K_n$ on the blocks and $K=K_0$ elsewhere in the graph chart.
It is $C^\infty$, all of its jets vanish at zero, and $K(t)>0$
for every $t\ne0$. Set $\gamma''=K$ and
$\gamma(0)=\gamma'(0)=0$. By the exact mass identities,
\begin{equation}\label{Bneg:eq:C-exact-tangents}
 \gamma'(t)=\gamma_0'(t)\quad(t\notin\bigcup B_n),\qquad
 \gamma'(x_{n,i})=s_{n,i}.
\end{equation}
Thus the graph slope stays uniformly small. Its actual geometric
curvature is $K/(1+\gamma'^2)^{3/2}$, so it has exactly one zero.

Each half-hole has the same exact quadratic graph and tangent-coordinate
calculation as in \Secref{Bneg:subsub:rational-blocks}, now with
\Eqref{Bneg:eq:C-holes}. Hence, after the fixed translation by $-o$,
a symmetric tangent-coordinate subarc of length $\mu_n\eqsim m_n$
satisfies
\begin{equation}\label{Bneg:eq:C-actual-arcs}
 \begin{split}
 y_{n,i}(t)&=z_{n,i}+t\tau_{n,i}+e_{n,i}(t),\qquad
 \dd\nu_{n,i}\ge\dd t,\\
 \tau_{n,i}&=A_n^{-1}v_{n,i}/|A_n^{-1}v_{n,i}|,\quad
 |e_{n,i}(t)|\le C\varepsilon_nm_n^2.
 \end{split}
\end{equation}
Here $e_{n,i}$ is in the actual normal direction and $\nu_{n,i}$
is the true-arclength restriction. These subarcs are disjoint, so
the eventual complete arclength satisfies
$\sigma\ge\sum_i\nu_{n,i}$ for every block.

\subsubsection{A positive completion and one interior origin fixed in advance}
\label{Bneg:subsec:C-completion}
Apply the germ-preserving completion in
\Secref{Bneg:subsec:negative-geometry} to the baseline, before its
translation or rescaling, to obtain a smooth strictly convex oval
$\Gamma_0$ with positive curvature off zero.

Fix any interior point $o$ of this reference oval before choosing
connector centres or phases. Equal total masses in every block give
the uniform position estimate
\begin{equation}\label{Bneg:eq:C-graph-close}
 \sup_{[-t_*,t_*]}|\gamma-\gamma_0|
 \le C\sum_{n\ge n_0}b_nD_n.
\end{equation}
For a fully included block, subtract its zero total mass in the
double integration. Its contribution is bounded by $C b_nD_n$.
A partially included block has the same bound, and the series
converges. The graph endpoint slopes and all curvature germs are
unchanged. Only the right endpoint height may differ.

Fix two smooth bumps on the complementary angle interval whose
tangent-vector integrals are linearly independent. Adding signed
multiples of these two bumps to the positive reference radius
corrects this displacement exactly. The coefficients have size
$O(\sum b_nD_n)$ and the supports stay away from endpoints.
For large fixed $n_0$ the corrected radius stays positive and
matches all endpoint germs. The completed curve is uniformly close
to $\Gamma_0$ in position, and its supporting-line distances from
$o$ stay positive: choose the error below half the minimum such
distance for $\Gamma_0$. Thus the same $o$ is interior for all
allowed connector choices and their limits. Translate once by $-o$.
The complement is
uniformly positive in curvature and introduces no additional zero.

\subsubsection{Uniform cap integrals with the direction supremum outside}
We verify the cap condition uniformly over the remaining connector
choices. The supporting direction stays fixed throughout each depth
integral.
\begin{lemma}\label{Bneg:lem:C-cap}
Every completed oval above has $\mathcal C_{p_*}(\Omega)<\infty$,
uniformly over the allowed connector centres.
\end{lemma}
\begin{proof}[Proof of \Lemref{Bneg:lem:C-cap}]
Use the cap calculation in \Secref{Bneg:subsec:cap-calculation}, now
with $q=p_*$, $a=1/2$ and $\varepsilon_n=e^{-m_n^{-s_0}}$.
Its hypotheses concern only the holes and lower bounds for $K$.
Here the holes have the same equal-spacing estimate
\Eqref{Bneg:eq:hole-packing}, and \Eqref{Bneg:eq:C-radius-tail}
pays its boundary errors. For fixed large $A_*$ and sufficiently
small $\beta$, their $(A_*+1)$-expansions are disjoint,
$K\ge\varepsilon_n$ on each expansion, and $K\ge c_0K_0$
outside the holes. Thus each fixed supporting contact has at most
one exceptional hole throughout the depth integral. The identical
defect and layer-cake calculation gives
\begin{equation}\label{Bneg:eq:C-cap-layercake}
 \begin{split}
 \int_0^{\delta_0}L_{x,e}(\delta)^q\frac{\dd\delta}{\delta}
 &\lesssim q\int_0^{L_{\max}}L^{q-1}
                 [1+\log^+(1/L)+L^{-a}]\,\dd L\\
 &\quad+q m_n^{-s_0}\int_0^{A_*m_n}L^{q-1}\,\dd L\\
 &\lesssim 1+m_n^{q-s_0}\lesssim1,
 \end{split}
\end{equation}
where the last term occurs for only that one hole. Here
$q>s_0>a$. Bounded graph slope compares normal depth and arclength
with the graph coordinates, so the estimate is uniform over the
original fixed normals. Shallow caps outside the central chart
lie in a compact positive-curvature part of the completed boundary
and have mass $O(\sqrt\delta)$. The remaining depths cost only a
fixed perimeter bound. This proves the assertion for the complete
original arclength.
\end{proof}

\subsection{Phase programming and a common test function}
\label{Bneg:sec:Ctest}

We use the fixed exponents
\begin{equation}\label{Bneg:Ctest:fixed-exponents}
 p_*>3,\qquad s_0=\frac{p_*+3}{2},\qquad a=\frac{1}{2}.
\end{equation}
Use the blocks and nodes of \Secref{Bneg:sec:geometry}, with
\Eqref{Bneg:eq:C-slopes}--\Eqref{Bneg:eq:C-holes}.
Write $v_{n,i}=(u_{n,i},v_{n,i}^{(2)})$.
The parameters satisfy
\begin{equation}\label{Bneg:Ctest:parameter-growth}
 \begin{gathered}
 m_n\eqsim b_n/N_n,\quad
 \log N_n\eqsim 2^{an}/n^2,\quad |\log b_n|=O(n),\\
 \log(1/m_n)=\log N_n+O(n)\quad\text{and}\quad
 \log\log(1/m_n)=O(n).
 \end{gathered}
\end{equation}
All geometric data are fixed independently of the later indicator tests.
In particular, \Eqref{Bneg:eq:C-matrix} gives
\begin{equation}\label{Bneg:Ctest:matrix-growth}
 \log\|A_n\|=O(t_n^{-a}+n)\quad\text{and}\quad
 |A_n^{-1}v_{n,i}|\eqsim Q_n.
\end{equation}
The actual positive movable connectors are exactly
\Eqref{Bneg:eq:C-residual}--\Eqref{Bneg:eq:C-movable-bump}.
Their first moments are $w_{n,i}c_{n,i}$, while
\Eqref{Bneg:eq:C-masses} preserves each prescribed slope.
Their uniform smoothness is \Eqref{Bneg:eq:C-jets}.
The fixed origin and the completion for every allowed centre array
are supplied by \Secref{Bneg:subsec:C-completion}.

We use the movable connector centres to prescribe the initial phases
of the actual curve points while preserving the cap estimate.
\begin{lemma}[Programming the initial original phases]
\label{Bneg:Ctest:phase-programming}
There is a fixed $C_*>0$ such that, with
\[
 k_{0,n}=\max\left\{0,
       \left\lceil\log_2(C_*Q_n/m_n)\right\rceil\right\},
 \qquad \mathfrak n_{n,i}=(-v_{n,i}^{(2)},u_{n,i}),
\]
every prescribed array
$\alpha_{n,i}\in\mathbb R/\mathbb Z$ can be realized on one
admissible limiting graph by
\begin{equation}\label{Bneg:Ctest:initial-phase}
 2^{k_{0,n}}\mathfrak n_{n,i}\cdot A_nz_{n,i}
      =\alpha_{n,i}\pmod1,\qquad
 z_{n,i}=(x_{n,i},\gamma(x_{n,i}))-o.
\end{equation}
Its completion is one fixed smooth strictly convex domain with
the fixed interior origin just specified.
\end{lemma}

\begin{proof}[Proof of \Lemref{Bneg:Ctest:phase-programming}]
For fixed incoming height and slope, double integration of the
connector gives
\[
 \gamma(x_{n,i})=\text{a fixed term}
                    +w_{n,i}(x_{n,i}-c_{n,i}).
\]
The fixed term includes the background. Hence the attainable height
interval has length at least
\[
 cD_nd_n/N_n\ge c b_nD_n/N_n^2.
\]
The outgoing slope is independent of $c_{n,i}$, and earlier node heights
are unchanged. Since
\begin{align*}
 \mathfrak n_{n,i}\cdot A_nz_{n,i}
 &=u_{n,i}M_n\gamma(x_{n,i})
  -(v_{n,i}^{(2)}+P_nu_{n,i})x_{n,i}
  -\mathfrak n_{n,i}\cdot A_no,
\end{align*}
the attainable interval on the left of \Eqref{Bneg:Ctest:initial-phase} has
length at least
\[
 c\,2^{k_{0,n}}Q_nD_n^{-1}\frac{b_nD_n}{N_n^2}
 \ge c C_*\frac{Q_n^2b_n}{m_nN_n^2}
 \ge c'C_*\ge2.
\]
Choose $C_*$ accordingly. The dependence on the current centre is affine
and continuous, so every residue is attained. Proceed from left to right
through the node intervals. The final connector restores the prescribed
outgoing slope and changes no earlier phase.

For an actual finite truncation ending at block $L$, retain the
baseline inside all blocks $n>L$, set
$\gamma(0)=\gamma'(0)=0$, and process the active blocks in the order
$L,L-1,\ldots,n_0$.  This is their spatial order from the origin outwards.
At each step the incoming height is known.  The incoming slope always
has its required baseline value, by preservation of every completed
block's total curvature mass.  Hence the finite construction satisfies
all its prescribed phases.

Each centre ranges over a compact interval. A diagonal subsequence
of the finite constructions converges at every fixed centre. On each
finite set of blocks this gives convergence of every curvature derivative,
while \Eqref{Bneg:eq:C-jets} makes the omitted tails uniformly small.
Thus the curvatures and their double integrals converge in $C^r$ for every
fixed $r$, including at zero. For each fixed $(n,i)$,
\[
 \exp\!\left(2\pi i\,2^{k_{0,n}}
             \mathfrak n_{n,i}\cdot A_nz_{n,i}\right)
       =e^{2\pi i\alpha_{n,i}}
\]
is a closed constraint and survives the limit for each fixed index.

The uniform completion in \Secref{Bneg:subsec:C-completion} applies
also to this limiting centre array. Its position bound
\Eqref{Bneg:eq:C-graph-close} keeps the same reference point $o$
strictly interior. The correction is confined to the complementary
arc and leaves every programmed graph point unchanged.
Thus the fixed translation by $-o$ in
\Eqref{Bneg:Ctest:initial-phase} remains valid, with no subsequent
change of origin or scale.
\end{proof}

The next elementary digit construction prescribes a finite sequence of
phases at separated dyadic times. It will select one common physical
scale for every contributing arc.
\begin{lemma}[Finite binary programming at separated times]
\label{Bneg:Ctest:binary-code}
For integers $h,T\ge1$ and arbitrary
$\vartheta_0,\ldots,\vartheta_{T-1}\in\mathbb R/\mathbb Z$, there is
$\alpha\in\mathbb R/\mathbb Z$ such that
\[
 \operatorname{dist}(2^{jh}\alpha-\vartheta_j,\mathbb Z)
       \le2^{-h}\qquad(0\le j<T).
\]
\end{lemma}

\begin{proof}[Proof of \Lemref{Bneg:Ctest:binary-code}]
Choose representatives in $[0,1)$, set
$a_j=\lfloor2^h\vartheta_j\rfloor$, and take
$\alpha=\sum_{j=0}^{T-1}a_j2^{-(j+1)h}$.  After multiplication by
$2^{jh}$ the preceding digit blocks are integers.  The remaining
fractional part belongs to the same interval
$[a_j2^{-h},(a_j+1)2^{-h}]$ as $\vartheta_j$, because the subsequent
finite tail is smaller than $2^{-h}$.  This proves the estimate,
including dyadic endpoint cases.  The iterates are one doubling
orbit, not independently chosen later phases.
\end{proof}

We now choose the phase array once and use
\Lemref{Bneg:Ctest:phase-programming} to fix the curve.  In one block
temporarily suppress $n$ and write $F=m^{-s_0}$.  Define
\begin{equation}\label{Bneg:Ctest:time-budget}
 \begin{aligned}
 K&=\left\lfloor\frac{F}{2\log2}\right\rfloor,
 &B&=K-k_0,\\
 h&=\lceil\log_2(B+2)\rceil,
 &g&=\left\lfloor\sqrt{B/(2h)}\right\rfloor,
 &T&=g^2.
 \end{aligned}
\end{equation}
For all sufficiently large blocks these are positive and
\begin{equation}\label{Bneg:Ctest:budget-comparisons}
 B\eqsim m^{-s_0},\quad h\eqsim\log(1/m),\quad
 g\eqsim\sqrt{B/h},\quad Th\le B/2.
\end{equation}
Indeed $k_0=O(\log Q+\log(1/m))=o(m^{-s_0})$, and the remaining
comparisons follow directly from the definitions.  Discard any finite
number of initial blocks needed for these estimates by increasing
the single starting generation before choosing the curve.

Let $u_0,\ldots,u_{T-1}$ be the regular $g$ by $g$ grid on
$\mathbb T^2=\mathbb R^2/\mathbb Z^2$.  It is a net of radius
$\sqrt2/(2g)$.  For each integer normal $\mathfrak n_i$ apply
\Lemref{Bneg:Ctest:binary-code} to
$\vartheta_{i,j}=\mathfrak n_i\cdot u_j\pmod1$, obtaining
$\alpha_i$.  Prescribe these arrays in
\Lemref{Bneg:Ctest:phase-programming} for all blocks.  On the resulting
one fixed curve set $\beta_i=\mathfrak n_i\cdot A z_i$.  At the genuine
original physical indices
\begin{equation}\label{Bneg:Ctest:physical-indices}
 k_j=k_0+jh,\qquad0\le j<T,
\end{equation}
we have
\begin{equation}\label{Bneg:Ctest:phase-following}
 \operatorname{dist}(2^{k_j}\beta_i-\mathfrak n_i\cdot u_j,
                         \mathbb Z)\le2^{-h}.
\end{equation}
All $k_j$ lie in $[k_0,K]$.  Their later phases are exactly the
doubling iterates of the programmed initial phases.  These prescribed
phases are part of the curve and will not be reset for any test function.

\smallskip\noindent\textbf{The actual central arcs and one common test function.}\space
The actual arcs are those of \Eqref{Bneg:eq:C-actual-arcs}, now with
one block index suppressed. Their parameter lengths are at least $cm$,
their masses are at least $cm$, and their errors satisfy
$|e_i(t)|\le C\varepsilon m^2$. They are disjoint true-arclength
restrictions, so $\sigma\ge\sum_{i=1}^N\nu_i$.

Choose one sufficiently large constant $C_0$ and set
\begin{equation}\label{Bneg:Ctest:stripe-width}
 \rho=\frac{C_0Q}{g}\quad\text{and}\quad
 \mathcal H=\bigcup_{i=1}^N
    \{u\in\mathbb T^2:
          \operatorname{dist}(\mathfrak n_i\cdot u,\mathbb Z)<\rho\}.
\end{equation}
The comparisons above and $s_0>3$ imply $N\rho\to0$, so in particular
$\rho<1/4$ for all sufficiently large blocks.  Each nonzero integer
character $u\mapsto\mathfrak n_i\cdot u\pmod1$ sends normalized Haar
measure on $\mathbb T^2$ to that on $\mathbb T$.  Its stripe has measure
$2\rho$, whence
\begin{equation}\label{Bneg:Ctest:source-density}
 |\mathcal H|\le2N\rho.
\end{equation}
The union bound applies regardless of dependence among directions.  Set the single periodic indicator $f(x)=\mathbf1_{\mathcal H}(Ax)$.

For each point $x$ choose a closest grid point $u_j$ to
$Ax\pmod{\mathbb Z^2}$.  The same $j$ then works for every direction:
by \Eqref{Bneg:Ctest:phase-following} and $|\mathfrak n_i|\le C_vQ$,
\[
 \operatorname{dist}(\mathfrak n_i\cdot Ax-2^{k_j}\beta_i,
                            \mathbb Z)
 \le C_vQ\frac{\sqrt2}{2g}+2^{-h}\le\rho/2.
\]
Here $C_0$ is fixed sufficiently large, and
$2^{-h}=o(Q/g)$ by \Eqref{Bneg:Ctest:budget-comparisons}.
Along the entire actual subarc in \Eqref{Bneg:eq:C-actual-arcs}, exact
tangency gives $\mathfrak n_i\cdot A\tau_i=0$, so
\[
 \mathfrak n_i\cdot A(x-2^{k_j}y_i(t))
 =\mathfrak n_i\cdot Ax-2^{k_j}\beta_i
                     -2^{k_j}\mathfrak n_i\cdot Ae_i(t).
\]
The final term is bounded uniformly in $i,j,t$ by
\begin{equation}\label{Bneg:Ctest:curved-error}
 C Q\|A\|2^K e^{-F}m^2
 \le C Q\|A\|m^2e^{-F/2}=o(\rho).
\end{equation}
Indeed,
\[
 \begin{gathered}
 \log\|A\|+\log Q+\log(1/m)+|\log\rho|=o(F),\\
 \log\bigl(CQ\|A\|m^2e^{-F/2}\bigr)=-F/2+o(F).
 \end{gathered}
\]
Thus the error in \Eqref{Bneg:Ctest:curved-error} is less than $\rho/4$ after
increasing the one initial generation. At the same index $k_j$ for every
$i$,
\[
 \operatorname{dist}\!\left(
   \mathfrak n_i\cdot A(x-2^{k_j}y_i(t)),\mathbb Z\right)
 <\rho,
 \qquad
 f(x-2^{k_j}y_i(t))=1.
\]
Consequently
\begin{equation}\label{Bneg:Ctest:common-value}
 (\sigma_{2^{k_j}}*f)(x)
 \ge\sum_{i=1}^N\nu_i(\mathbb R^2)
 \ge cNm\ge c'b.
\end{equation}
The entire curved arcs contribute at one common physical index. Choosing
the least nearest-grid index in a tie gives a measurable selector.

\smallskip\noindent\textbf{Bounded planar indicator tests.}\space
For fixed block data define the bounded planar indicators
\[
 E_R=\{x\in[-R,R]^2:Ax\pmod{\mathbb Z^2}\in\mathcal H\},
 \qquad\text{and}\qquad f_R=\one_{E_R},
\]
and take
\[
 D_{\rm loc}=1+2^K
       \sup_{y\in\operatorname{supp}\sigma}\|y\|_\infty.
\]
The measure is compactly supported and the selected family is finite.
Thus \Eqref{Bneg:Ctest:common-value} holds for $f_R$ on
$[-R+D_{\rm loc},R-D_{\rm loc}]^2$. Tiling by the fixed lattice
$A^{-1}\mathbb Z^2$ gives
\[
 \frac{|E_R|}{(2R)^2}\longrightarrow|\mathcal H|,
 \qquad\text{and}\qquad
 \frac{(2R-2D_{\rm loc})^2}{(2R)^2}\longrightarrow1.
\]
Only $O(R)$ fundamental cells meet the boundary, so the common cell
Jacobian cancels. Testing at level $c'b/2$ and letting $R\to\infty$ yields
\begin{equation}\label{Bneg:Ctest:RWT-lower}
 \begin{split}
 \|M_\sigma\|_{L^p\to L^{p,\infty}}^{\mathrm{restricted}}
 &\ge c_pb|\mathcal H|^{-1/p}\\
 &\ge c_p'b(N\rho)^{-1/p}.
 \end{split}
\end{equation}
For each block, a finite $R$ realizes any smaller ratio. All averages act
on the same $f_R$, and their finite physical family is contained in the
unchanged bilateral supremum.

\smallskip\noindent\textbf{Every fixed exponent $p_*>3$.}\space
The estimate for the periodic density retains its necessary $Q$ factor:
\begin{equation}\label{Bneg:Ctest:density-powers}
 \begin{split}
 N\rho&\le C NQ\sqrt{h/B}\\
 &\le C b^{s_0/2}N^{(3-s_0)/2}[\log(1/m)]^{1/2}.
 \end{split}
\end{equation}
Consequently \Eqref{Bneg:Ctest:RWT-lower}, applied with $p=p_*$, gives
\begin{equation}\label{Bneg:Ctest:ten-thirds-divergence}
 \begin{split}
 \|M_\sigma\|_{L^{p_*}\to L^{p_*,\infty}}^{\mathrm{restricted}}
 &\ge c_{p_*} b_n^{1-s_0/(2p_*)}N_n^{(s_0-3)/(2p_*)}
                    [\log(1/m_n)]^{-1/(2p_*)}\\
 &=c_{p_*} b_n^{3(p_*-1)/(4p_*)}N_n^{(p_*-3)/(4p_*)}
                    [\log(1/m_n)]^{-1/(2p_*)}
 \longrightarrow\infty.
 \end{split}
\end{equation}
Indeed its logarithm is
$[(p_*-3)/(4p_*)]\log N_n+O_{p_*}(n)$ by
\Eqref{Bneg:Ctest:parameter-growth}. Here
$\log N_n\eqsim2^{an}/n^2\gg n$ and $p_*>3$.
This proves restricted weak, ordinary weak and strong failure at each
fixed $p_*>3$. The curve may depend on $p_*$, but is independent of
the block index and all later indicator tests.

Together with \Lemref{Bneg:lem:C-cap}, this proves the second
case of \Thmref{Bneg:thm:general-negative}.

\medskip
\noindent\textbf{Conflicts of Interest}\space
The authors declare that there are no conflicts of interest.
\par\medskip
\noindent{\bf AI disclosure.}
ChatGPT and Codex assisted with drafting, mathematical checks and
references. This is a research draft. The named authors must verify
the mathematics, authorship information and disclosures before
submission.

\end{document}